\documentclass{article}	

\usepackage{tikz}
\usepackage{tikz-cd}
\usepackage{amsmath}
\usepackage{amsfonts}
\usepackage{amssymb}
\usepackage{amsthm}
\usepackage{mathrsfs}
\usepackage{mathtools}
\usepackage{color}
\usepackage{graphicx}
\usepackage{hyperref}
\usepackage{pinlabel}  
\usepackage{graphicx}  
\usepackage[all]{xy}
\usepackage{amscd}
\usepackage{authblk}
\usepackage{ulem}
\usepackage{stmaryrd}
\usepackage{braket}   
\usetikzlibrary{arrows}

\author{Hiroshi Kihara}
\title{Quillen equivalences between the model categories of smooth spaces, simplicial sets, and arc-gengerated spaces}

\affil{Center for Mathematical Sciences, University of Aizu, 
 Tsuruga, Ikki-machi, Aizu-Wakamatsu City, Fukushima, 965-8580, Japan (kihara@u-aizu.ac.jp)}

\newtheoremstyle{mytheoremstyle} 
{1.25\topsep}                
{2.0\topsep}                 
{\it}
{}                           
{\bfseries}                  
{}                           
{1em}                        
{\thmname{#1}\ \thmnumber{#2}\  \thmnote{#3}\hspace{0pt}} 

\newtheorem{thm}{Theorem}[section]    
\newtheorem{lem}[thm]{Lemma}          
\newtheorem{prop}[thm]{Proposition}
\newtheorem{axiom}{Axiom}
\newtheorem{problem}[thm]{Problem}

\theoremstyle{definition}
\newtheorem{defn}[thm]{Definition}    
\newtheorem{rem}[thm]{Remark}             
\newtheorem{exa}[thm]{Example}

\theoremstyle{mytheoremstyle}

\newcommand{\bvec}[1]{\mbox{\boldmath $#1$}}

\newcommand*{\longhookrightarrow}{\ensuremath{\lhook\joinrel\relbar\joinrel\rightarrow}}

\makeatletter
\newcommand*{\relrelbarsep}{.386ex}
\newcommand*{\relrelbar}{%
	\mathrel{%
		\mathpalette\@relrelbar\relrelbarsep
	}%
}
\newcommand*{\@relrelbar}[2]{%
	\raise#2\hbox to 0pt{$\m@th#1\relbar$\hss}%
	\lower#2\hbox{$\m@th#1\relbar$}%
}
\providecommand*{\rightrightarrowsfill@}{
	\arrowfill@\relrelbar\relrelbar\rightrightarrows
}
\providecommand*{\leftleftarrowsfill@}{
	\arrowfill@\leftleftarrows\relrelbar\relrelbar
}
\providecommand*{\xrightrightarrows}[2][]{
	\ext@arrow 0359\rightrightarrowsfill@{#1}{#2}%
}
\providecommand*{\xleftleftarrows}[2][]{
	\ext@arrow 3095\leftleftarrowsfill@{#1}{#2}%
}
\makeatother

\def\abb{{\mathbb{A}}}

\def\qbb{{\mathbb{Q}}}
\def\rbb{{\mathbb{R}}}

\def\zbb{{\mathbb{Z}}}

\def\acal{{\mathcal{A}}}
\def\bcal{{\mathcal{B}}}
\def\ccal{{\mathcal{C}}}
\def\dcal{{\mathcal{D}}}
\def\ecal{{\mathcal{E}}}

\def\ical{{\mathcal{I}}}
\def\jcal{{\mathcal{J}}}
\def\kcal{{\mathcal{K}}}

\def\mcal{{\mathcal{M}}}

\def\scal{{\mathcal{S}}}
\def\tcal{{\mathcal{T}}}
\def\ucal{{\mathcal{U}}}
\def\vcal{{\mathcal{V}}}
\def\wcal{{\mathcal{W}}}

\def\czero{{\mathcal{C}^0}}

\def\Hom{{\mathrm{Hom}}}

\def\colim{{\mathrm{colim}}}

\DeclareMathAlphabet{\mathpzc}{OT1}{pzc}{m}{it}

\begin{document}
 \maketitle
 
 \begin{abstract}
 In the preceding paper, we have constructed a compactly generated model structure on the category $\dcal$ of diffeological spaces together with the adjoint pairs $|\ |_\dcal : \scal \rightleftarrows \dcal : S^\dcal$ and $\tilde{\cdot} : \dcal \rightleftarrows \ccal^0 : R$, where $\scal$ and $\ccal^0$ denote the category of simplicial sets and that of arc-generated spaces, respectively. In this paper, we show that $(|\ |_\dcal, S^\dcal)$ and $(\tilde{\cdot}, R)$ are pairs of Quillen equivalences. Since our approach developed in the preceding paper applies to the category $\ccal h$ of Chen spaces as well, $\ccal h$ is also a compactly generated model category. We also show that the adjoint pair $\mathfrak{S}_\mathfrak{o} : \ccal h \rightleftarrows \dcal : \mathfrak{Ch}^\sharp$ introduced by Stacey is a pair of Quillen equivalences.
 \end{abstract}
 
\section{Introduction}
Let $\scal$ and $\ccal^0$ denote the category of simplicial sets and that of arc-generated spaces, respectively; recall that a topological space $X$ is called {\sl arc-generated} if the topology of $X$ is final for $C^0_X \coloneqq \text{\{continuous curves from } \rbb \text{ to } X \text{\}}$. In the preceding paper \cite{origin}, we have constructed a compactly generated model structure on the category $\dcal$ of diffeological spaces, introducing the adjoint pairs
$$
|\ |_\dcal : \scal \rightleftarrows \dcal : S^\dcal \ \text{and} \ \tilde{\cdot} : \dcal \rightleftarrows \ccal^0 : R,
$$
whose composite is just the adjoint pair
$$
|\ | : \scal \rightleftarrows \ccal^0 : S
$$
of the topological realization and singular functors (see Section 2).
\par\indent
In the following theorem, which is the main result of this paper, we establish the Quillen equivalences between the model categories $\scal$, $\dcal$, and $\ccal^0$ via $(|\ |_\dcal, S^\dcal)$ and $(\tilde{\cdot}, R)$. See \cite{MP, Hi} for model categories and Quillen equivalences, \cite{K, GJ, MP} for the model structure of $\scal$, and Remark \ref{arc} for the model structure of $\ccal^0$.
\begin{thm}\label{Quillenequiv}
	\begin{itemize}
		\item[(1)]
		$|\ |_{\dcal}: \scal \rightleftarrows \dcal: S^{\dcal}$ is a pair of Quillen equivalences.
		\item[(2)]
		$\tilde{\cdot}: \dcal \rightleftarrows \ccal^{0}: R$ is a pair of Quillen equivalences.
	\end{itemize}
\end{thm}
Theorem 1.3 in \cite{origin}, which is recorded as Theorem \ref{originmain} in this paper, and Theorem \ref{Quillenequiv} are the central organizational theorems in smooth homotopy theory. For the proof of Theorem \ref{Quillenequiv}, we develope the theory of singular homology for diffeological spaces, which is naturally isomorphic to Hector's singular homology \cite{H} and Iglesias-Zemmour's cubic homology \cite[6.61]{I}. 
\par\indent
Since every diffeological space is fibrant, the class $\wcal_\dcal$ of diffeological spaces having the $\dcal$-homotopy type of a cofibrant object is important from the homotopical point of view. We thus study diffeological spaces in $\wcal_\dcal$ and ones not in $\wcal_\dcal$ using Theorem \ref{Quillenequiv}. Our study sheds new light on Christensen-Wu's comparison of the smooth and topological fundemental groups (\cite[Examples 3.12 and 3.20]{CW}).
\par\indent
Though Theorem \ref{Quillenequiv} shows that the Quillen homotopy categories of $\dcal$ and $\ccal^0$ are equivalent via the derived functors of $\tilde{\cdot}$ and $R$, it does not ensure that $S^\dcal \dcal (X,Y)$ is homotopy equivalent to $S\ccal^0 (\tilde{X}, \tilde{Y})$ for $X,Y \in \dcal$. For important diffeological spaces such as infinite dimensional $C^\infty$-manifolds, to find a sufficient condition under which these two complexes are homotopy equivalent is a much more subtle problem which is closely related to geometry. We will address this problem in the succeeding paper. The notions introduced to study diffeological spaces having the $\dcal$-homotopy type of a cofibrant object are relevant to this problem.
\par\indent
Next, we turn to Chen spaces (\cite{BH}). Since our construction of a model structure (\cite{origin}) applies to the category $\ccal h$ of Chen spaces as well (Theorem 6.3), we have a result analogous to Theorem \ref{Quillenequiv} (Theorem \ref{analogue}). Further, we establish the direct Quillen equivalence between $\ccal h$ and $\dcal$ via the adjoint pair $\mathfrak{S}_\mathfrak{o} : \ccal h \rightleftarrows \dcal : \mathfrak{Ch}^\sharp$ introduced by Stacey \cite{St}.
\begin{thm}\label{ChD}
	$\mathfrak{S}_\mathfrak{o} : \ccal h \rightleftarrows \dcal : \mathfrak{Ch}^\sharp$ is a pair of Quillen equivalences.
\end{thm}
In Section 2, we recall the basic notions and results from \cite{origin}. In Section 3, we develope the singular homology theory for diffeological spaces. In Section 4, we prove Theorem \ref{Quillenequiv} using the results of Section 3 (see Remark \ref{proof}). In Section 5, we discuss diffeological spaces having the $\dcal$-homotopy type of a cofibrant object. In Section 6, we endow the category $\ccal h$ with a model structure and prove Theorem \ref{ChD}.

\section{Compactly generated model structure on the category $\dcal$}
We make a review on the construction of a compactly generated model structure on the category $\dcal$. Then, we observe that the adjoint pairs $|\ |_\dcal : \scal \rightleftarrows \dcal : S^\dcal$ and $\tilde{\cdot} : \dcal \rightleftarrows \ccal^0 : R$ are Quillen pairs.
\par\indent
Let us begin with the definition of a diffeological space. A {\sl parametrization} of a set $X$ is a (set-theoretic) map $p: U \longrightarrow X$, where $U$ is an open subset of $\rbb^{n}$ for some $n$.
\begin{defn}\label{diffeological}
	\begin{itemize}
		\item[(1)] A {\sl diffeological space} is a set $X$ together with a specified set $D_X$ of parametrizations of $X$ satisfying the following conditions:
		\begin{itemize}
			\item[(i)](Covering)  Every constant parametrization $p:U\longrightarrow X$ is in $D_X$.
			\item[(ii)](Locality) Let $p :U\longrightarrow X$ be a parametrization such that there exists an open cover $\{U_i\}$ of $U$ satisfying $p|_{U_i}\in D_X$. Then, $p$ is in $D_X$.
			\item[(iii)](Smooth compatibility) Let $p:U\longrightarrow X$ be in $D_X$. Then, for every $n \geq 0$, every open set $V$ of $\rbb^{n}$ and every smooth map $F  :V\longrightarrow U$, $p\circ F$ is in $D_X$.
		\end{itemize}
		The set $D_X$ is called the {\sl diffeology} of $X$, and its elements are called {\sl plots}.
		\item[(2)] Let $X=(X,D_X)$ and $Y=(Y,D_Y)$ be diffeological spaces, and let $f  :X\longrightarrow Y$ be a (set-theoretic) map. We say that $f$ is {\sl smooth} if for any $p\in D_X$, \ $f\circ p\in D_Y$. Then, diffeological spaces and smooth maps form the category $\dcal$.
	\end{itemize}
\end{defn}
The category $\dcal$ of diffeological spaces has the obvious underlying set functor. See \cite[pp. 230-233]{FK} for initial and final structures, and initial and final morphisms with respect to the underlying set functor.
\par\indent
The convenient properties of $\dcal$ are summarized as the following proposition.
\begin{prop}\label{conven}
	\begin{itemize}
		\item[(1)] The category ${\dcal}$ has initial and final structures with respect to the underlying set functor. In particular, ${\dcal}$ is complete and cocomplete. Further, the class of $\dcal$-embeddings (i.e., injective initial morphisms) is closed under pushouts and transfinite composites.
		\item[(2)] The category $\mathcal{D}$ is cartesian closed.
		\item[(3)] The underlying set functor $\dcal \longrightarrow Set$
		is factored as the underlying topological space functor 
		$\widetilde{\cdot}:\dcal \longrightarrow \czero$
		followed by the underlying set functor
		$\czero \longrightarrow Set$, where $\czero$ denotes the category of arc-generated spaces.
		Further, the functor
		$\widetilde{\cdot}:\dcal \longrightarrow \czero$
		has a right adjoint
		$R:\czero \longrightarrow \dcal$.
	\end{itemize}
	\begin{proof}
		See \cite[p. 90]{CSW}, \cite[pp. 35-36]{I}, and \cite[Propositions 2.1 and 2.13]{origin}.
	\end{proof}
\end{prop}
The notions of diffeological subspace and quotient diffeological space are defined in the standard manner (\cite[Definition 2.2]{origin}). A $\dcal$-quotient map is defined to be a surjective final morphism.
\par\indent
For Part 3 of Proposition \ref{conven}, recall that the underlying topological space $\tilde{A}$ of a diffeological space $A = (A, D_A)$ is defined to be the set $A$ endowed with the final topology for $D_A$ and that $R$ assigns to an arc-generated space $X$ to the set $X$ endowed with the diffeology
$$
D_{RX} = \text{\{continuous parametrizations of } X \text{\}.}
$$
It is obvious that $\tilde{\cdot} \circ R = Id_{\ccal^0}$.
\par\indent
The principal part of our construction of a model structure on $\dcal$ is the construction of good diffeologies on the sets
$$
\Delta^p=\{(x_0,\ldots,x_p)\in\mathbb{R}^{p+1} \ |\ \underset{i}{\sum} x_i = 1,\ x_i\geq 0 \}\ \ \ (p\geq 0)
$$
which enables us to define weak equivalences, fibrations, and cofibrations and to verify the model axioms (see Definition \ref{WFC} and Theorem \ref{originmain}). The required properties of the diffeologies on $\Delta^{p} \ (p \geq 0)$ are expressed in the following four axioms:
\begin{axiom}
	The underlying topological space of $\Delta^p$ is the topological standard $p$-simplex for $p\geq 0$.
\end{axiom}
Recall that
$f:\Delta^p \longrightarrow \Delta^q$
is an {\sl affine map} if $f$ preserves convex combinations.
\begin{axiom}
	Any affine map $f:\Delta^p\longrightarrow \Delta^q$ is smooth.
\end{axiom}
Let $\scal$ denote the category of simplicial sets. For $K \in \scal$, the {\sl simplex category} $\Delta\downarrow K$ is defined to be the full subcategory of the overcategory $\scal \downarrow K$ consisting of maps $\sigma : \Delta[n] \rightarrow K$. Consider the diagram $\Delta\downarrow K \longrightarrow \mathcal{S}$ sending $\sigma :\Delta[n] \longrightarrow K$ to $\Delta[n]$, and note that the colimit $\underset{\mathrm{\Delta\downarrow} K}{\mathrm{colim}} \ \Delta[n]$ is naturally isomorphic to $K$ (\cite[p. 7]{GJ}). Using Axiom 2, we can define the {\sl realization functor}
$$
|\ |_{\dcal}: \mathcal{S}\longrightarrow \mathcal{D}
$$
by $|K|_{\mathcal{D}}= \underset{\mathrm{\Delta\downarrow} K}{\mathrm{colim}} \ \Delta^n$.
\begin{axiom}
	Let $K$ be a subcomplex of $\Delta[p]$. Then the canonical smooth injection
	$$\left| K \right|_{\dcal} \longhookrightarrow \Delta^p$$
	is a $\dcal$-embedding.
\end{axiom}
The notion of a deformation retract in $\dcal$ is defined in the same manner as in the category of topological spaces by using the unit interval $I=[0,\ 1]$ endowed with a diffeology via the canonical bijection with $\Delta^{1}$ (\cite[Section 2.4]{origin}).
The {\sl $k^{th}$ horn} of $\Delta^p$ is a diffeological subspace of $\Delta^p$ defined by
\begin{eqnarray*}
	\Lambda^p_k &=&  \{(x_0,\ldots,x_p)\in\Delta^p \ |\ x_i=0 \hbox{ for some }i\neq k\}.
\end{eqnarray*}
\begin{axiom}
	The $k^{th}$ horn $\Lambda^p_k$ is a deformation retract of $\Delta^p$ in $\mathcal{D}$ for $p \geq 1$ and $0 \leq k \leq p$.
\end{axiom}
For a subset $A$ of the affine $p$-space $\abb^{p} = \{(x_0, \ldots, x_p) \in \rbb^{p+1} \ | \ \sum x_i = 1 \}$, $A_{\mathrm{sub}}$ denotes the set $A$ endowed with the sub-diffeology of $\abb^{p} \ (\cong \rbb^{p})$.
The diffeological space $\Delta^{p}_{\mathrm{sub}}$, used in \cite{H} to study diffeological spaces by homotopical means, satisfies neither Axiom 3 nor 4 for $p \geq 2$ (\cite[Proposition A.2]{origin}). Thus, we must construct a new diffeology on $\Delta^p$, at least for $p \geq 2$. Let us introduce such diffeologies on $\Delta^p$.
\par\indent
Let $(i)$ denote the vertex $(0, \ldots, \underset{(i)}{1}, \ldots, 0)$ of $\Delta^p$, and let $d^i$ denote the affine map from $\Delta^{p-1}$ to $\Delta^p$, defined by
\begin{eqnarray*}
	d^i((k))= \left \{
	\begin{array}{ll}
		(k) & \text{for} \ k<i,\\
		(k+1)& \text{for} \ k\geq i.
	\end{array}
	\right.
\end{eqnarray*} 

\begin{defn}\label{simplices}
	We define the {\sl standard $p$-simplices} $\Delta^p$ ($p\geq 0$) inductively. Set $\Delta^p=\Delta_{\mathrm{sub}}^p$ for $p\leq 1$. Suppose that the diffeologies on $\Delta^k$ ($k<p$) are defined. 
	We define the map 
	\begin{eqnarray*}
		\varphi_i: \Delta^{p-1}\times [0,1) & \longrightarrow & \Delta^p
	\end{eqnarray*}
	by $\varphi_{i}(x, t) = (1-t)(i)+td^{i}(x)$, and endow $\Delta^p$ with the final structure for the maps $\varphi_{0}, \ldots, \varphi_{p}$.
\end{defn}

\begin{prop}\label{axioms}
	The standard $p$-simplices $\Delta^p (p \geq 0)$ in Definition \ref{simplices} satisfies Axioms 1-4.
\end{prop}
\begin{proof}
	See \cite[Propositions 3.2, 5.1, 7.1, and 8.1]{origin}.
\end{proof}

Without explicit mention, the symbol $\Delta^p$ denotes the standard $p$-simplex defined in Definition \ref{simplices} and a subset of $\Delta^p$ is endowed with the sub-diffeology of $\Delta^p$; see \cite[Lemmas 3.1 and 4.2, and Remark A.4]{origin} for a comparison of the diffeologies of $\Delta^{p}$ and $\Delta^{p}_{\mathrm{sub}}$.
\par\indent
Since the diffeology of $\Delta^p$ is the sub-diffeology of $\abb^{p}$ for $p \leq 1$, our notion of a deformation retract in $\dcal$ coincides with the ordinary notion of a deformation retract in the theory of diffeological spaces (\cite[p. 110]{I} and \cite[Remark 2.20]{origin}).
\par\indent

By Axiom 2, we can define the singular complex $S^{\dcal}X$ of a diffeological space $X$ to have smooth maps $\sigma : \Delta^p \longrightarrow X$ as $p$-simplices, thereby defining the {\sl singular functor} $S^{\mathcal{D}} :\mathcal{D} \longrightarrow \mathcal{S}$.
\par\indent
Recall the adjoint pair
$$
|\ | : \scal \rightleftarrows \ccal^0 : S
$$
of the topological realization functor $|\ |$ and the topological singular functor $S$ (\cite[Section 9.1]{origin}). As usual, $\dot{\Delta}^p$ denotes the boundary of $\Delta^p$.

\begin{lem}\label{adjoint}
	\begin{itemize}
		\item[(1)] $|\ |_{\dcal}: \scal \rightleftarrows \dcal: S^{\dcal}$ is an adjoint pair.
		\item[(2)] The composite of the two adjoint pairs 
		$$
		|\ |_{\dcal}: \scal \rightleftarrows \dcal: S^{\dcal} \text{ and } \tilde{\cdot}: \dcal \rightleftarrows \ccal^{0}: R
		$$ 
		is just the adjoint pair
		$$
		|\ | : \scal \rightleftarrows \ccal^0 : S.
		$$
		\item[(3)] Let $X$ be a diffeological space. Then the singular complex $S^{\dcal}X$ is a subcomplex of the topological singular complex $S\widetilde{X}$.
		\item[(4)] The underlying topological spaces of $\Lambda^p_k$ and $\dot{\Delta}^p$ are subspaces of the topological standard $p$-simplex.
	\end{itemize}
	\begin{proof}
		See \cite[Proposition 9.1 and Lemma 9.2]{origin} for Parts 1 and 2. The smooth map $id: X \longrightarrow R\widetilde{X}$ induces the inclusion $S^{\dcal}X \longhookrightarrow S^{\dcal}R\widetilde{X} = S\widetilde{X}$ (Part 2), which implies Part 3. Part 4 follows immediately from Proposition \ref{axioms} (Axiom 3) and Part 2.
	\end{proof}
\end{lem}

\begin{defn}\label{WFC}
	Define a map $f :X\longrightarrow Y$ in $\mathcal{D}$ to be
	\begin{itemize}
		\item[$(1)$]
		a {\sl weak equivalence} if $S^{\mathcal{D}} f:S^{\mathcal{D}} X\longrightarrow S^{\mathcal{D}} Y$ is a weak equivalence in the category of simplicial sets,
		\item[$(2)$]
		a {\sl fibration} if the map $f$ has the right lifting property with respect to the inclusions $\Lambda^p_k \longhookrightarrow\Delta^p$ for all $p>0$ and $0\leq k\leq p$, and
		\item[$(3)$]
		a {\sl cofibration} if the map $f$ has the left lifting property with respect to all maps that are both fibrations and weak equivalences.
	\end{itemize}
\end{defn}
Define the sets $\ical$ and $\jcal$ of morphisms of $\dcal$ by
\begin{eqnarray*}
	\ical & = & \{\dot{\Delta}^{p} \longhookrightarrow \Delta^{p} \ | \ p\geq 0 \},\\
	\jcal & = & \{\Lambda^{p}_{k} \longhookrightarrow \Delta^{p} \ | \ p>0,\ 0 \leq k \leq p \}.
\end{eqnarray*}

\begin{lem}\label{generating}
	Let $f  :X\longrightarrow Y$ be a morphism of $\mathcal{D}$.
	\begin{itemize}
		\item[$(1)$]
		The following conditions are equivalent:
		\begin{itemize}
			\item[$(\mathrm{i})$]
			$f  :X\longrightarrow Y$ is a fibration;
			\item[$(\mathrm{ii})$]
			$S^{\mathcal{D}} f : S^{\mathcal{D}} X\longrightarrow S^{\mathcal{D}} Y$ is a fibration;
			\item[$(\mathrm{iii})$]
			$f$ has the right lifting property with respect to $\jcal$.
		\end{itemize}
		\item[$(2)$]
		The following conditions are equivalent:
		\begin{itemize}
			\item[$(\mathrm{i})$]
			$f  :X\longrightarrow Y$ is both a fibration and a weak equivalence;
			\item[$(\mathrm{ii})$]
			$S^{\mathcal{D}} f : S^{\mathcal{D}} X\longrightarrow S^{\mathcal{D}} Y$ is both a fibration and weak equivalence;
			\item[$(\mathrm{iii})$]
			$f$ has the right lifting property with respect to the $\ical$.
		\end{itemize}
	\end{itemize}
	\begin{proof}
		\cite[Lemma 9.3]{origin}.
	\end{proof}
\end{lem}

Using these results, the following theorem, which is the main result of \cite{origin}, is obtained.
\begin{thm}\label{originmain}
	With Definition \ref{WFC}, $\mathcal{D}$ is a compactly generated model category whose object is always fibrant. $\ical$ and $\jcal$ are the sets of generating cofibrations and generating trival cofibrations respectively.
	\begin{proof}
		\cite[Theorem 1.3]{origin}.
	\end{proof}
\end{thm}
	
The following theorem shows that the singular complex $S^{\dcal}X$ captures smooth homotopical properties of $X$, and that our model structure on $\dcal$ organizes the smooth homotopy theory of diffeological spaces. See \cite[Section 3.1]{CW} or \cite[Chapter 5]{I} for the smooth homotopy groups $\pi^{\dcal}_{p}(X, x)$ of a pointed diffeological space $(X, x)$, and see \cite[p. 25]{GJ} for the homotopy groups $\pi_p(K, x)$ of a pointed Kan complex $(K, x)$.

\begin{thm}\label{homotopygp}
	Let $(X, x)$ be a pointed diffeological space. Then, there exists a natural bijection
	$$
	\varTheta_{X} : \pi^{\dcal}_{p}(X, x) \longrightarrow \pi_{p}(S^{\dcal}X, x) \ \  \text{for} \ \ p \geq 0,
	$$
	that is an isomorphism of groups for $p > 0$. \\
	\begin{proof}
		\cite[Theorem 1.4]{origin}.
	\end{proof}
\end{thm}
From Theorem \ref{homotopygp}, we see that weak equivalences are just smooth maps inducing isomorphisms on smooth homotopy groups.

For the proof of Theorem \ref{Quillenequiv}, we also need the following lemma.
\begin{lem}\label{delta}
	\begin{itemize}
		\item[(1)] Let $f  :\Delta^r \longrightarrow \Delta^p\times I$ be an affine map (i.e., a map preserving convex combinations). Then, $f$ is smooth.
		\item[(2)] The horn $\Lambda^p_k$, and hence, the standard $p$-simplex $\Delta^p$ is constractible in $\dcal$.
		\item[(3)] The map $id : \Delta^p \longrightarrow \Delta_{{\rm sub}}^p$ is smooth, which restricts to the diffemorphism $id : \Delta^p - {\rm sk}_{p-2} \ \Delta^p \xrightarrow[\cong]{} (\Delta^p - {\rm sk}_{p-2} \ \Delta^p)_{\rm sub}$, where ${\rm sk}_{p-2} \ \Delta^p$ denotes the $(p-2)$-skeleton of $\Delta^p$.
		\item[(4)] $\dot{\Delta}^p$ is a deformation retract of $\Delta^p - \{b_p\}$ in $\dcal$, where $b_p$ is the barycenter of $\Delta^p$.
	\end{itemize}
	\begin{proof}
		Part 1 follows immediately from Proposition \ref{axioms} (Axiom 2). See \cite[Corollary 8.3, Lemmas 3.1 and 4.2, and Proposition 6.2]{origin} for Parts 2-4.
	\end{proof}
\end{lem}

\begin{rem}\label{ingredients}
	Theorem \ref{homotopygp} and Lemma \ref{delta}(3)-(4) are established for the standard simplices in Definition \ref{simplices}. The proofs of Lemmas \ref{adjoint}-\ref{generating}, Theorem \ref{originmain}, and Lemma \ref{delta}(1)-(2) are constructed using only the convenient properties of the category $\dcal$ (Proposition \ref{conven}) and Axioms 1-4 for the standard simplices (Proposition \ref{axioms}).
\end{rem}

\begin{rem}\label{arc}
	The basic facts on arc-generated spaces are summarized in \cite[Section 2.2]{origin}. The category $\ccal^{0}$ is a compactly generated model category whose object is always fibrant (cf. \cite[Notation 1.1.2 and Example 11.1.8]{Hi}), which is shown in the same manner as in the cases of the category of topological spaces (\cite[Section 8]{DS}) and that of compactly generated Hausdorff spaces (\cite[Proposition 9.2 in Chapter I]{GJ}). We can easily see that the adjoint pair $I :\czero \rightleftarrows \tcal : \alpha$ (\cite[Lemma 2.7]{origin}) is a pair of Quillen equivalences.
\end{rem}

For the adjoint pairs $(|\ |_\dcal, S^\dcal)$ and $(\tilde{\cdot}, R)$, we have the following result.

\begin{lem}\label{Quillenpairs}
	$(|\ |_{\mathcal{D}},S^{\mathcal{D}})$ and $(\widetilde{\cdot},R)$ are Quillen pairs.
	\begin{proof}
		Since $S^{\mathcal{D}}$ preserves both fibrations and trivial fibrations (Lemma \ref{generating}), $(|\ |_{\mathcal{D}},S^{\mathcal{D}})$ is a Quillen pair. Since $R$ also preserves both fibrations and trivial fibrations (Proposition \ref{conven}(3) and Lemma \ref{adjoint}(4)), $(\widetilde{\cdot},R)$ is a Quillen pair.
	\end{proof}
\end{lem}

The following lemma states that the composite of the adjoint pairs $(|\ |_\dcal, S^\dcal)$ and $(\tilde{\cdot}, R)$ is a pair of Quillen equivalences.
\begin{lem}\label{clequiv}
	$|\ |: \scal \rightleftarrows \ccal^{0}:S$ is a pair of Quillen equivalences.
	\begin{proof}
		This result is shown in the same manner as in the case of $|\ |:\scal \rightleftarrows \ucal:S$, where $\ucal$ is the category $\tcal$ of topological spaces or the category $\kcal$ of compactly generated Hausdorff spaces; see \cite[Theorem 16.1]{May} and \cite[Theorem 11.4 in Chapter I]{GJ} for the cases $\tcal$ and $\kcal$ respectively.
	\end{proof}
\end{lem}

\begin{rem}\label{proof}
	As mentioned in the proof of Proposition \ref{clequiv}, two proofs of the fact that $(|\ |, S)$ is a pair of Quillen equivalences are known (\cite[Theorem 16.1]{May}, \cite[Theorem 11.4 in Chapter I]{GJ}). We prove Part 1 of Theorem \ref{Quillenequiv} in a way analogous to the proof of \cite[Theorem 16.1]{May}. For this, we develope the singular homology theory for diffeological spaces in Section 3.
	\par\indent
	The other proof is based on Quillen's result that the topological realization of a Kan fibration is a Serre fibration (\cite[Theorem 10.10 in Chapter I]{GJ}). For the realization functor $|\ |_{\mathcal{D}}:\scal\longrightarrow\mathcal{D}$, the analogous result does not seem to hold; note that $|\ |_{\mathcal{D}}$ does not preserve finite products (in fact, $|\Delta[1]\times\Delta[1]|_\dcal$ is not the product $|\Delta[1]|_\dcal \times |\Delta[1]|_\dcal$ but a diffeological space obtained by patching together two copies of $\Delta^2$). Thus, it seems that Theorem \ref{Quillenequiv}(1) can not be proved according to the idea of the proof of \cite[Theorem 11.4 in Chapter I]{GJ}.
\end{rem}

\section{Singular homology of a diffeological space.}
In this section, we introduce the singular homology of a diffeological space and establish the results used in the proof of Theorem \ref{Quillenequiv}.
\par\indent
Let us begin by recalling the homology groups of a simplicial set.
\par\indent
For a category $\ccal$, $s\ccal$ denotes the category of simplicial objects in $\ccal$ (cf. \cite[Section 2]{May}). Let $\zbb$-$mod$ denote the category of $\zbb$-modules, and let $Kom_{\geq0}(\zbb$-$mod)$ denote the category of non-negatively graded chain complexes of $\zbb$-modules.
The free $\zbb$-module functor $\zbb\cdot: Set \longrightarrow \zbb\text{-}mod$ extends to the functor $\mathbb{Z}\cdot:\scal = sSet \longrightarrow s\zbb\text{-}mod$. The composite
\begin{center}
	$\scal \xrightarrow{\zbb\cdot} s\zbb\text{-}mod \longrightarrow Kom_{\geq 0}(\zbb\text{-}mod)$
\end{center} 
is also denoted by $\zbb\cdot$, where the functor $s\zbb\text{-}mod$ $\longrightarrow Kom_{\geq 0}(\zbb\text{-}mod)$ is defined by assigning to a simplicial $\zbb$-module $M$ the chain complex 
\begin{center}
	$\cdots \xrightarrow{\partial_1} M_1 \xrightarrow{\partial_0} M_0 \longrightarrow 0 $
\end{center} 
with $\partial_n = \sum_{i=0}^{n} (-1)^i d_i$.
\par\indent
A {\sl simplicial pair} $(K,L)$ consists of a simplicial set $K$ and a simplicial subset $L$ of $K$. Define the {\sl homology} $H_{\ast}(K,L)$ of a simplicial pair $(K,L)$ by 
$$
H_{\ast}(K,L)=H_{\ast}(\zbb(K)/\zbb(L)).
$$
The homology of a simplicial pair $(K,\phi)$ is usually referred to as the homology of a simplicial set $K$, written $H_{\ast}(K)$.
\par\indent
A {\sl diffeological pair} $(X,A)$ consists of a diffelogical space $X$ and a diffeological subspace $A$ of $X$. Define the {\sl singular homology} $H_{\ast}(X,A)$ of a 
diffeological pair 
$(X,A)$ by
$$
H_{\ast}(X,A)=H_{\ast}(S^{\dcal} X,S^{\dcal} A).
$$
The singular homology of a diffeological pair $(X,\phi)$ is usually referred to as the {\sl singular homology} of a diffeological space $X$, written $H_{\ast}(X)$.
\par\indent
We show that the singular homology for diffeological pairs has the following desirable properties.

\begin{prop}\label{homologytheory}
	\begin{itemize}
		\item[(1)]
		(Exactness) For a diffeological pair $(X,A)$, there exists a natural exact sequence
		\begin{center}
			\begin{tikzcd}
				&\cdots
				\arrow[d, phantom, ""{coordinate, name=Z}]
				&\phantom{H_p(X,A)} \arrow[dll,
				"\partial",
				rounded corners,
				to path={ -- ([xshift=2ex]\tikztostart.east)
					|- (Z) [near end]\tikztonodes
					-| ([xshift=-2ex]\tikztotarget.west)
					-- (\tikztotarget)}] \\
				H_p(A) \arrow[r]
				& H_p(X) \arrow[r]
				\arrow[d, phantom, ""{coordinate, name=Z}]
				& H_p(X,A) \arrow[dll,
				"\partial",
				rounded corners,
				to path={ -- ([xshift=2ex]\tikztostart.east)
					|- (Z) [near end]\tikztonodes
					-| ([xshift=-2ex]\tikztotarget.west)
					-- (\tikztotarget)}] \\
				H_{p-1}(A) \arrow[r]
				&  \cdots
				&  
			\end{tikzcd}	
		\end{center}
		\item[(2)]
		(Excision) If $(X,A)$ is a diffeological pair and $U$ is a diffeological subspace of $X$ such that $\overline{U}\subset \mathring{A}$, then the inclusion $(X-U,A-U)\longhookrightarrow (X,A)$ induces an isomorphism $H_{\ast}(X-U,A-U)\xrightarrow[\ \cong\ ]{} H_{\ast}(X,A)$.
		\item[(3)]
		(Homotopy) Let $f,g:(X,A)\longrightarrow (Y,B)$ be maps of diffeological pairs. If $f\simeq g$, then the induced maps $f_{\ast},g_{\ast}:H_{\ast}(X,A)\longrightarrow H_{\ast}(Y,B)$ coincide.
		\item[(4)]
		(Dimension)
		Let $\ast$ be a terminal object of $\mathcal{D}$. Then
		\[
		H_p({\ast})=
		\begin{cases}
		\zbb & p=0 \\
		0 & p\neq 0.
		\end{cases}
		\]
		\item[(5)]
		(Additivity) If $X$ is the coproduct of diffeological spaces $\{X_i \}$, then the inclusions $X_i \longrightarrow X$ induce an isomorphism
		$$
		\underset{i}{\bigoplus} H_\ast(X_i) \xrightarrow[\cong]{} H_\ast(X).
		$$
	\end{itemize}	
\end{prop}

\begin{proof}
	Parts 1, 4, and	5 are easily verified. For the proof of Parts 2 and 3, we note that $\zbb S^{\dcal}X$ is a chain subcomplex of $\zbb S\widetilde{X}$ (Lemma \ref{adjoint}(3)), and verify that the relevant chain map and chain homotopies on $\zbb S\widetilde{X}$ restricts to ones on $\zbb S^{\dcal}X$.	
	\item[(3)] To show that $\tilde{f}_\sharp \simeq \tilde{g}_\sharp:\zbb S\widetilde{X} \longrightarrow \zbb S\widetilde{Y}$, we use the chain homotopy $D$ on $\zbb S\widetilde{X}$ defined using the prism decomposition of $\Delta^p \times I$ (\cite[p. 45]{G}), which restricts to a chain homotopy on $\zbb S^\dcal X$ by Lemma \ref{delta}(1).
	\item[(2)] To verify the excision axiom for the topological singular homology, we use the chain map $Sd$ on $\zbb S\widetilde{X}$ and the chain homotopy $T$ on $\zbb S\widetilde{X}$ connecting $1_{\zbb S\widetilde{X}}$ and $Sd$, where $Sd$ and $T$ are defined using the barycentric subdivision of $\Delta^p$ and the triangulation of $\Delta^p \times I$ connecting the trivial triangulation of $\Delta^p$ and the barycentric subdivision of $\Delta^p$ (\cite[Section 17 in Chapter IV]{Br}). Thus, $Sd$ and $T$ restricts to a chain map and a chain homotopy on $\zbb S^{\dcal}X$ respectively by Proposition \ref{axioms}(Axiom 2) and Lemma \ref{delta}(1).
\end{proof}
The Mayer-Vietoris exact sequence and the homology exact sequence of a triple are also deduced from Proposition \ref{homologytheory} via the usual arguments (cf. \cite[p. 37]{Selick}); these exact sequences can be also derived from the consideration of the singular chain complexes.
\par\indent
For a $\zbb$-module $M$ and $n \geq 0$, let $M[n]$ denote the graded module with $M[n]_n=M$ and $M[n]_i=0$ ($i\neq n$).

\begin{lem}\label{homologicallemma}
	Let $b$ denote the barycenter of $\Delta^p$. Then
	\[
	H_{\ast}(\mathring{\Delta}^p,\mathring{\Delta}^p-\{b\})\cong\zbb[p].
	\]
	\begin{proof}
		Define the reduced homology $\widetilde{H}_\ast(X)$ of a diffeological space $X$ in the same manner as the reduced homology of a topological space (cf. \cite[p. 38]{Selick}). Since $\mathring{\Delta}^p$ is contractible in $\dcal$ (Lemma \ref{delta}(3)), we have only to show that $\widetilde{H}_\ast(\mathring{\Delta}^p - \{b\}) \cong \zbb[p-1]$ (see Proposition \ref{homologytheory}(1)).
		\par\indent
		Since $\mathring{\Delta}^p - \{b\}$ contains a diffeological subspace which is diffeomorphic to $S^{p-1}$ as a deformation retract in $\dcal$, we see that $\widetilde{H}_\ast(\mathring{\Delta}^p - \{b\}) \cong \widetilde{H}_\ast (S^{p-1}) \cong \zbb [p-1]$ by Proposition \ref{homologytheory} and the Mayer-Vietoris exact sequence.
	\end{proof}
\end{lem}

Let us introduce the notions of a $CW$-complex in $\dcal$ and its cellular homology.
\par\indent
Recall that $\dcal$ is a cofibrantly generated model category having $\ical = \{\dot{\Delta}^p \longhookrightarrow \Delta^p\ |\ p \geq 0\}$ as the set of generating cofibrations. A $CW$-complex in $\dcal$, introduced in Definition \ref{CW} is a special type of an $\ical$-cell complex.
\begin{defn}\label{CW}
	A {\sl $CW$-complex} $X$ in $\dcal$ is a diffeological space $X$ which is the colimit of a sequence in $\dcal$
	$$
	X^0\overset{i_1}{\longhookrightarrow}X^1\overset{i_2}{\longhookrightarrow}X^2\overset{i_3}{\longhookrightarrow}\cdots
	$$
	such that each $i_n$ fits into a pushout diagram of the form
	\begin{center}
		\begin{tikzcd}
			\underset{\lambda\in\Lambda_n}{\coprod} \dot{\Delta}^n_\lambda \arrow{d} \arrow[hook]{r} & \underset{\lambda\in\Lambda_n}{\coprod} \Delta^n_\lambda \arrow{d}\\
			X^{n-1} \arrow[hook]{r}{i_n} & X^n.
		\end{tikzcd}
	\end{center}
\end{defn}
Note that $CW$-complexes in $\dcal$ are cofibrant objects in $\dcal$. Let us see that the realizations of simplicial sets form an important class of $CW$-complexes in $\dcal$.

\begin{lem}\label{realization}
	Let $K$ be a simplicial set. Then, the realization $|K|_\dcal$ is a $CW$-complex in $\dcal$ having one $n$-cell for each non-degenerate $n$-simplex of $K$.
	\begin{proof}
		Let $K^n$ denote the $n$-skeleton of $K$. Then, $K$ is the colimit of the sequence
		$$
		K^0 \xhookrightarrow{\ \ i_1\ \ } K^1 \xhookrightarrow{\ \ i_2\ \ } K^2 \xhookrightarrow{\ \ i_3\ \ } \cdots
		$$
		and the pushout diagram
		\begin{center}
			\begin{tikzcd}
				\underset{\lambda \in NK_n}{\coprod} \dot{\Delta}[n]_\lambda \arrow{d} \arrow[hook]{r} & \underset{\lambda \in NK_n}{\coprod} \Delta[n]_\lambda \arrow{d}\\
				K^{n-1} \arrow[hook]{r}{i_n} & K_n
			\end{tikzcd}
		\end{center}
		exists for $n \geq 0$, where $NK_n$ is the set of non-degenerate $n$-simplices of $K$. Applying the realization functor $|\ |_\dcal$, we obtain the presentation of the $CW$-structure of $|K|_\dcal$ (see Lemma \ref{adjoint}(1), Proposition \ref{axioms}(Axiom 3) and \cite[p. 8]{GJ}).
	\end{proof}
\end{lem}

For a $CW$-complex $X$ in $\dcal$, define the {\sl cellular chain complex} $(C(X), \partial_C)$ by
\begin{center}
	\begin{equation*}
	\abovedisplayskip=-10pt   	
		\begin{split}
			C_n(X) =& H_n (X^n, X^{n-1}), \\
			\partial_C =& \text{the connecting homomorphism coming from the } \\
			&\text{homology exact sequence of the triple } (X^n, X^{n-1}, X^{n-2}).
		\end{split}
	\end{equation*}
\end{center}
(See \cite[p. 39]{Selick} for the verification of $\partial^2_C = 0$.) The homology $H_\ast C(X)$ is called the {\sl cellular homology} of $X$.

\begin{lem}\label{cellular}
	Let $X$ be a $CW$-complex in $\dcal$.
	\begin{itemize}
		\item[(1)] The $n^{th}$ component $C_n(X)$ is a free $\zbb$-module on $\{n \text{-cells of } X\}$.
		\item[(2)] $H_\ast C(X)$ is isomorphic to $H_\ast(X)$.
	\end{itemize}
	\begin{proof}
		Express $X$ as in Definition \ref{CW}, and let $b_\lambda$ denote the barycenter of $\Delta^n_\lambda$. The point of $X^n$ corresponding to $b_\lambda$ is also denoted by $b_\lambda$. From Lemma \ref{delta}(4) and \cite[Lemma 6.3(2) and 7.2]{origin}, we observe that
		\begin{center}
			\begin{tikzcd}
				X^n - \{b_\lambda | \lambda \in \Lambda_n\} &=& X^{n-1} \underset{\underset{\lambda \in \Lambda_n}{\coprod}\dot{\Delta}^n_\lambda}{\cup} \underset{\lambda\in\Lambda_n}{\coprod} \Delta^n_\lambda - \{b_\lambda\} \nonumber \\[-7mm]
				&\simeq& X^{n-1}. \ \ \ \ \ \ \ \ \ \ \ \ \ \ \ \ \ \ \ \ \ \ \ \ \ \ \ \ \ 
			\end{tikzcd}
		\end{center}
		By Proposition \ref{homologytheory} and Lemma \ref{homologicallemma}, we can prove the result in the same manner as in the topological case (see the proof of \cite[Theorem 5.4.2]{Selick}).
	\end{proof}
\end{lem}

\begin{prop}\label{homologyisone}
	For a simplicial set $K$, there exists a natural isomorphism of graded modules
	$$
	H_\ast (K) \cong H_\ast (|K|_\dcal).
	$$
	\begin{proof}
		Since $(|\ |_\dcal, S^\dcal)$ is an adjoint pair (Lemma \ref{adjoint}(1)), we have the natural map $i_K : K \longrightarrow S^\dcal |K|_\dcal$, which induces the homomorphism
		$$
		H_\ast (K) \xrightarrow{\ \ i_{K_\ast}\ \ } H_\ast (S^\dcal |K|_\dcal) = H_\ast (|K|_\dcal). 
		$$
		With the help of Lemma \ref{cellular}, we can show that $i_{K_\ast}$ is an ismorphism by an argument similar to that in the proof of Proposition 16.2(i) in \cite{May}.
	\end{proof}
\end{prop}

We have proven all the results on the singular homology and the realization of a simplicial set which are used in the proof of Theorem \ref{Quillenequiv}(1). The following are additional remarks.

\begin{rem}\label{cohomology} (1) Let $\pi$ be an abelian group. For a diffeological pair $(X, A)$, the {\sl singular homology} $H_\ast(X, A;\pi)$ {\sl with coefficients in} $\pi$ is defined by
	$$
	H_\ast(X, A;\pi) = H_\ast(\zbb S^\dcal X/\zbb S^\dcal A\otimes \pi ).
	$$
	Then, Proposition \ref{homologytheory} generalizes to the case of a general coefficient group $\pi$ (see \cite[p. 183]{Br} and the proof of Proposition \ref{homologytheory}).
	\par\indent
	For a diffeological pair $(X, A)$, the {\sl singular cohomology} $H^{\ast}(X, A;\pi)$ {\sl with coefficients in} $\pi$ is defined by 
	$$
	H^{\ast}(X, A;\pi) = H^{\ast} \Hom(\zbb S^{\dcal}X/\zbb S^{\dcal}A, \pi),
	$$	
	which has desirable properties analogons to Proposition \ref{homologytheory} (see \cite[pp. 285-286]{Br} and the proof of Proposition \ref{homologytheory}).
	\item[(2)] The universal coefficient theorems obviously hold for the singular homology and cohomology of diffeological spaces (cf. \cite[pp. 281-285]{Br}). Since the Eilenberg-Zilber theorem holds for simplicial sets (\cite[Corollary 29.6]{May}), the Kunneth theorems hold for the singular homology and cohomology of diffeological spaces, and hence, the cross products on homology and cohomology, and the cup product on cohomology are defined and satisfy the same formulas as those in the topological case. 
	\item[(3)] Since the functor $S^{\dcal}$ preserves fibrations (Lemma \ref{generating}), the Serre spectral sequence works for fibrations in $\dcal$ (\cite[Section 32]{May}).
\end{rem}

\begin{rem}\label{coincide}
	Hector \cite{H} defined the singular complex of a diffeological space $X$ using $\Delta^p_\text{sub} (p \geq 0)$, which is denoted by $S^\dcal_\text{sub}(X)$. Iglesias-Zemmour defined the complex $\bvec{\mathsf{C}}_\star (X)$ of reduced groups of cubic chains for a diffeological space $X$, and called its homology the cubic homology of $X$ (\cite[pp. 182-183]{I}).
	\par\indent
	Similar to the topological case, we can see that $\zbb S^\dcal (X)$, $\zbb S^\dcal_\text{sub} (X)$, and $\bvec{\mathsf{C}}_\star (X)$ are naturally homotopy equivalent (see \cite{EM}), and hence that our and Hector's singular homologies and the cubic homology are naturally isomorphic (see the footnote of \cite[p. 183]{I}).
\end{rem}

\section{Quillen equivalences between $\scal$, $\dcal$, and $\ccal^0$}
In this section, we prove Theorem \ref{Quillenequiv}. 

\subsection{Proof of Part 1 of Theorem \ref{Quillenequiv}}
	We can reduce the proof of Theorem \ref{Quillenequiv}(1) by the following lemma on the Quillen pair $(|\ |_{\dcal}, S^{\dcal})$.
	\begin{lem}\label{1streduction}
		The following are equivalent:
		\begin{itemize}
			\item[(i)]
			$(|\ |_{D},S^{\mathcal{D}})$ is a pair of Quillen equivalences.
			\item[(ii)]
			For $K\in\scal$ and $X\in\mathcal{D}$, the natural maps
			\[
			i_K:K\longrightarrow S^{\mathcal{D}}|K|_{\mathcal{D}}\hbox{  and  } p_X:|S^{\mathcal{D}} X|_{\mathcal{D}}\longrightarrow X
			\]
			are weak equivalences in $\scal$ and $\mathcal{D}$ respectively.
			
			\item[(iii)]
			For $K\in\scal$, the natural map $i_K:K\longrightarrow S^{\mathcal{D}}|K|_{\mathcal{D}}$ is a weak equivalence in $\scal$.
		\end{itemize}
		\begin{proof}
			Note that every object of $\scal$ is cofibrant and that every object of $\dcal$ is fibrant. Then the implications (i)$\implies$(ii)$\implies$(iii) are obvious.\\
			(iii)$\implies$(i) Let $X$ be a diffeological space. For a map $f:K\longrightarrow S^{\mathcal{D}} X$, consider the commutative diagram
			\begin{center}
				\begin{tikzpicture}
				\node at (0,0)
				{\begin{tikzcd}[column sep=large]
					K \arrow{r}{f} \arrow[d,"i_K"'] & S^{\mathcal{D}} X \arrow[d,"i_{S^{\mathcal{D}} X}"] \\
					S^{\mathcal{D}} |K|_{\mathcal{D}} \arrow{r}{S^{\mathcal{D}}|f|_{\dcal}} \arrow[dr,"S^{\mathcal{D}} f\hat{\phantom{g}}"']& 
					S^{\mathcal{D}}|S^{\mathcal{D}} X \arrow[d,"{S^{\mathcal{D}} p_X}"]|_{\dcal}\\
					& S^{\mathcal{D}} X,
					\end{tikzcd}};
				\draw[->] (1.75,1.15)--(2.45,1.15)--(2.45,-1.15)--(1.75,-1.15);
				\node [right] at (2.5,0) {$1_{S^{\mathcal{D}} X}$};
				\end{tikzpicture}
			\end{center}
			where $f\hat{\phantom{g}}:|K|_{\mathcal{D}}\longrightarrow X$ is the left adjunct of $f:K\longrightarrow S^{\mathcal{D}} X$. Since $i_K$ is a weak equivalence, we have the equivalences
			\begin{center}
				\begin{tabular}{rcl}
					$f$ is a weak equivalence &$\Leftrightarrow$& $S^{\mathcal{D}} f\hat{\phantom{g}}$ is a weak equivalence\\
					&$\Leftrightarrow$& 
					$f\hat{\phantom{g}}$ is a weak equivalence,
				\end{tabular}
			\end{center}
			which shows that $(|\ |_{\mathcal{D}},S^{\mathcal{D}})$ is a pair of Quillen equivalences.
		\end{proof}
	\end{lem}
	
	Thus, we have only to show that for $K\in\scal$, the natural map $i_K:K\longrightarrow S^{\mathcal{D}}|K|_{\mathcal{D}}$ is a weak equivalence in $\scal$. Since $|\ |_\dcal$ and $S^\dcal$ preserve coproducts, we may assume that $K$ is connected.
	
	Let us further reduce the proof of Part 1 of Theorem \ref{Quillenequiv} to simpler cases.
	
	\begin{lem}\label{2ndreduction}
		Consider the following conditions:
		\begin{itemize}
			\item[(i)]
			$f:K\longrightarrow K'$ is a weak equivalence in $\scal$.
			\item[(ii)]
			$|f|_{\dcal}:|K|_{\dcal}\longrightarrow|K'|_{\dcal}$ is a weak equivalence in $\mathcal{D}$.
			\item[(iii)]
			$S^{\mathcal{D}}|f|_{\mathcal{D}}:S^{\mathcal{D}}|K|_{\mathcal{D}}\longrightarrow S^{\mathcal{D}}|K'|_{\mathcal{D}}$ is a weak equivalence in $\scal$.
		\end{itemize}
		Then, the implications (i)$\implies$(ii)$\implies$(iii) hold.
		\begin{proof}
			Since every object of $\scal$ is cofibrant, the implication $(i) \Longrightarrow (ii)$ holds by Lemma \ref{Quillenpairs} and \cite[Corollary 7.7.2]{Hi}. The implication $(ii) \Longrightarrow (iii)$ is immediate from the definition of a weak equivalence in $\dcal$.
		\end{proof}
	\end{lem}
	For a connected simplicial set $K$, consider weak equivalences in $\scal$
	$$
	K \ \hookrightarrow \ L \ \hookleftarrow \ M,
	$$
	where $L$ is a fibrant approximation of $K$ and $M$ is a minimal subcomplex of $L$ (\cite[$\mathsection$9]{May}). Then, by Lemma \ref{2ndreduction}, we have only to show that for a Kan complex $K$ with only one vertex, 
	$i_{K}: K \longrightarrow S^{\dcal}|K|_{\dcal}$ is a weak equivalence in $\scal$.
	
	A simplicial map $\varpi:L\longrightarrow K$ is called a {\sl covering} if $\varpi$ is a fiber bundle with discrete fiber. For a Kan complex $K$ with only one vertex, a covering $\varpi:\tilde{K}\longrightarrow K$ with $\tilde{K}$ $1$-connected is constructed in a natural way, which is called the {\sl universal covering} of $K$ (\cite[Definition 16.4]{May}).
	
	A smooth map $p:E\longrightarrow B$ is called a {\sl fiber bundle} if any $b\in B$ has an open neighborhood $U$ such that $p^{-1}(U)\cong U\times F$ over $U$ for some $F\in\mathcal{D}$. A smooth map $p:E\longrightarrow B$ is called a {\sl covering} if $p$ is a fiber bundle with discrete fiber.
	
	\begin{lem}\label{covering}
		Let $K$ be a Kan complex with only one vertex, and let $\varpi:\tilde{K}\longrightarrow K$ be the universal covering of $K$.
		\begin{itemize}
			\item[(1)]
			$|\varpi|_{\mathcal{D}}:|\tilde{K}|_{\mathcal{D}}\longrightarrow |K|_{\mathcal{D}}$ is a covering.
			\item[(2)]
			$S^{\mathcal{D}}|\varpi|_{\mathcal{D}}:S^{\mathcal{D}}|\tilde{K}|_{\mathcal{D}}\longrightarrow S^{\mathcal{D}}|K|_{\mathcal{D}}$ is a covering.
		\end{itemize}
		\begin{proof}
			(1) We need the following three facts on $|\varpi|_{\dcal}: |\tilde{K}|_{\dcal} \longrightarrow |K|_{\dcal}$. \\ \\
			{\sl Fact 1: $|\varpi|_\dcal$ is a $\dcal$-quotient map.}
			\par\indent
			Recall that the realization $|L|_{\dcal}$ of a simplicial set $L$ is defined by $|L|_{\dcal} = \underset{\Delta\downarrow L}{\colim} \ \Delta^n$ (Section 2). Thus, we have the commutative diagram
			\begin{equation*}
				\begin{tikzcd}
					\underset{\mathrm{Ob}(\Delta\downarrow\widetilde{K})}{\coprod}{\Delta}^{n}\arrow{d}\arrow{r} & {|\tilde{K}|}_\dcal \arrow{d}{\mathrm{|\varpi|_\dcal}}\\
					\underset{\mathrm{Ob}(\Delta\downarrow K)}{\coprod}{\Delta}^{n}\arrow[r] & {|K|}_\dcal.
				\end{tikzcd}
			\end{equation*}
			consisting of the canonical projections, where $\underset{\mathrm{Ob}(\Delta\downarrow L)}{\coprod}\Delta^n$ denotes the coproduct of the standard simplices indexed by the objects of $\Delta\downarrow{L}$. Since the horizontal arrows and the left verical arrow are $\dcal$-quotient maps, $|\varpi|_\dcal$ is also a $\dcal$-quotient map. \\ \\
			
			\noindent{\sl Fact 2: Each $g \in \pi_1 (K)$ acts on $|\tilde{K}|_\dcal$ as an automorphism in the overcategory $\dcal\downarrow|K|_{\dcal}$.}
			\par\indent
			The desired action of $\pi_1 (K)$ is obtained from the obvious action of $\pi_1 (K)$ on $\tilde{K}$ via the realization functor. \\ \\
			
			\noindent{\sl Fact 3: $|\varpi|_{\dcal}$ is a topological covering.}
			\par\indent
			The fact follows from Lemma \ref{adjoint}(2) and \cite[Theorem 4.2 in Chapter III]{GZ}. \\
			\par\indent
			Let us prove the result using Facts 1-3.
			For $x\in |K|_\dcal$, choose an open neighborhood $U$ which is topologically weakly contractible (\cite[1.9 in Chapter III]{GZ}). Then
			$$
			|\varpi|_D : |\varpi|^{-1}_{\dcal}(U) \longrightarrow U
			$$
			is a $\dcal$-quotient map by Fact 1 and \cite[Lemma 2.4(2)]{origin}, and $|\varpi|^{-1}_{\dcal}(U)$ is isomorphic to $\underset{g\in\pi_{1} (K)}{\coprod}U_g$, where each $U_g$ is an open diffeological subspace of $|\tilde{K}|_\dcal$ which is topologically isomorphic to $U$ (Fact 3). Since each $g\in\pi_1 (K)$ acts on $|\varpi|^{-1}_{\dcal}(U)$ as an automorphism of $\dcal \downarrow U$ (Fact 2), each $U_g$ is diffeomorphic to $U$ via $|\varpi|_\dcal$.\\ \\
			(2) Let $\sigma:\Delta^n \longrightarrow |K|_\dcal$ be a singular simplex, and consider the following lifting problem
			$$ 
			\begin{tikzcd}
			& {|\tilde{K}|_\dcal} \arrow[d,"\mathrm{|\varpi|_\dcal}"]\\
			\Delta^n \arrow[ur,dashed] \arrow[r,"\sigma"]& {|K|_\dcal}.
			\end{tikzcd}
			$$
			By Proposition \ref{axioms}(Axiom 1) and Fact 3 in the proof of Part 1, there is a continuous solution $\tau: \Delta^n \longrightarrow |\tilde{K}|_\dcal$ and $\{\tau \cdot g \}_{g\in\pi_1 (K)}$ is the set of all continuous solutions. Part 1 shows that the elements $\tau \cdot g$ are smooth, which completes the proof.
		\end{proof}
	\end{lem}
	Consider the morphism between homotopy exact sequences induced by the morphism of coverings
	\begin{center}
		\begin{tikzcd}
			\pi \arrow[d] \arrow[r,"="'] &\pi \arrow[d]\\
			\tilde{K} \arrow[d,"\varpi"'] \arrow[r,"i_{\tilde{K}}"]& S^{\mathcal{D}}|\tilde{K}|_{\mathcal{D}} \arrow[d, "S^{\mathcal{D}}|\varpi|_{\mathcal{D}}"] \\
			K \arrow[r,"i_K"]& S^{\mathcal{D}} |K|_{\mathcal{D}}.
		\end{tikzcd}
	\end{center}
	For the proof of Theorem \ref{Quillenequiv}(1), we then have only to show that for a  1-connected Kan complex $K$, $i_{K}: K \longrightarrow S^{\dcal}|K|_{\dcal}$ is a weak equivalence in $\scal$.
	Thus, the following lemma and Proposition \ref{homologyisone} complete the proof of Theorem \ref{Quillenequiv}(1) via the Whitehead theorem (\cite[Theorem 13.9]{May}).
	
	\begin{lem}\label{1-connected}
		If $K$ is a $1$-connected Kan complex, then $S^{\mathcal{D}}|K|_{\mathcal{D}}$ is also a $1$-connected Kan complex.
		
		\begin{proof}
			By replacing $K$ with its minimal subcomplex, we may assume that $K$ has only one non-degenerate simplex of dimension $\leq 1$ (Lemma \ref{2ndreduction}). Then $|K|_{\mathcal{D}}$ is a diffeological space obtained from the $0$-simplex $\ast$ by attaching simplices of dimension $\geq 2$ (Lemma \ref{realization}). Since $\pi_{1}(S^{\dcal}|K|_{\dcal}) \cong \pi^{\dcal}_{1}(|K|_{\dcal})$ (Theorem \ref{homotopygp}), we can easily see that $\pi_{1}(S^{\dcal}|K|_{\dcal}) = 0$ by \cite[Lemma 9.6]{origin}, the transversality theorem (\cite[(14.7)]{BJ}), and the argument in the proof of Lemma \ref{cellular}.
		\end{proof}
	\end{lem}

\subsection{Proof of Part 2 of Thorem \ref{Quillenequiv}}
	In this subsection, we prove Part 2 of Theorem \ref{Quillenequiv}. We need the following lemma.
	\par\indent
	For a model category $\mcal$, $\mcal_c$ and $\mcal_f$ denote the full subcategories of $\mcal$ consisting of cofibrant objects and fibrant objects, respectively.
	
	\begin{lem}\label{Quillenpairreduction}
		Let $\acal$, $\bcal$, and $\ccal$ be model categories, and
		\begin{center}
			\begin{tikzpicture} 
			\node [below] at (0,0.25) {\begin{tikzcd}
				{\acal\ }\arrow[r, shift left, "F"]
				& \bcal \arrow[l, shift left, "G"] \arrow[r, shift left, "L"]
				& {\ \ccal} \arrow[l, shift left, "R"]
				\end{tikzcd}};
			\draw[->] (-1.5,0) -- (-1.5,0.25)--(1.5,0.25)-- (1.5,0);
			\draw[<-] (-1.5,-.5) -- (-1.5,-.75)--(1.5,-.75)-- (1.5,-0.5);
			\node at (0,.5) {$LF$};
			\node at (0,-1.) {$GR$};
			\end{tikzpicture}
		\end{center}
		a diagram of functors such that $(F,G)$ and $(L,R)$ are Quillen pairs. Suppose that $(F,G)$ and $(LF,GR)$ are pairs of Quillen equivalences, and that $\acal_c=\acal$ and $\bcal_f=\bcal$. Then $(L,R)$ is also a pair of Quillen equivalences.
		\begin{proof}
			Let $X$ be a cofibrant object of $\bcal$ and $U$ a fibrant object of $\ccal$. Let $f:X\longrightarrow RU$ be a morphism of $\bcal$. Then we show that $f:X\longrightarrow RU$ is a weak equivalence if and only if the left adjunct $f\ \hat{}:LX\longrightarrow U$ is a weak equivalence.\\
			\par\noindent
			{\sl Step 1: The case of $X=FA$ with $A\in\acal$.}
			
			Since $(F,G)$ is a pair of Quillen equivalences, $f:FA\longrightarrow RU$ is a weak equivalence if and only if the right adjunct $f\ \check{}: A\longrightarrow GRU$ is a weak equivalence. Since $(LF,GR)$ is a pair of Quillen equivalences, $f\ \check{}: A\longrightarrow GRU$ is a weak equivalence if and only if $f\ \hat{}: LFA\longrightarrow U$ is a weak equivalence.\\
			\par\noindent
			{\sl Step 2: The case of a general cofibrant object $X$.}
			
			Set $A=GX$. Then the canonical map $p_X:FA\longrightarrow X$, which is the left adjunct of $1_{GX}: A = GX \longrightarrow GX$, is a weak equivalence between cofibrant objects in $\bcal$, and hence, $Lp_X:LFA\longrightarrow LX$ is a weak equivalence by \cite[Corollary 7.7.2]{Hi}. Thus, we have the equivalences
			\begin{center}
				\begin{tabular}{rl}
					& $f:X\longrightarrow RU$ is a weak equivalence in $\bcal$ \\
					$\Longleftrightarrow$
					& the composite $FA\xrightarrow[]{\ p_X\ }X\xrightarrow[]{\ f\ }RU$ is a weak equivalence in $\bcal$ \\
					$\Longleftrightarrow$
					& the composite $LFA\xrightarrow[]{\ Lp_X\ }LX\xrightarrow[]{\ f\ \hat{}\ }U$ is a weak equivalence in $\ccal$ \\
					$\Longleftrightarrow$
					& $f\ \hat{}: LX\longrightarrow U$ is a weak equivalence in $\ccal$, 
				\end{tabular}
			\end{center}
			using Step 1.
		\end{proof}
	\end{lem}
	
	\begin{proof}[\sl Proof of Theorem \ref{Quillenequiv}(2).] Recall that $\scal_c = \scal$ and that $\dcal_f = \dcal$ (Theorem \ref{originmain}). From Lemmas \ref{Quillenpairs}, \ref{adjoint}(2), and \ref{clequiv}, and Theorem \ref{Quillenequiv}(1), we obtain the result via Lemma \ref{Quillenpairreduction}.
	\end{proof}

\section{Diffeological spaces having the $\dcal$-homotopy type of a cofibrant object}	
	Using the unit interval $I\ (\cong\Delta^1)$, the $\dcal$-homotopy type (or smooth homotopy type) of a diffeological space is defined in the obvious manner (see \cite[Section 2.4]{origin}). In this section, we study diffeological spaces having the $\dcal$-homotopy type of a cofibrant object, and give various diffeological spaces that do not have the $\dcal$-homotopy type of a cofibrant object.
	\par\indent
	We define the subclass $\wcal_{\ccal^0}$ of $\ccal^0$ by
	$$
	\wcal_{\ccal^0} = \{X \in \ccal^0 \ | \ X \ \text{has the homotopy type of a cofibrant arc-generated space} \}.
	$$
	Note that $\wcal_{\ccal^0} = \wcal \cap \ccal^0$, where $\wcal$ denote the class of topological spaces having the homotopy type of a $CW$-complex; the class $\wcal$ was introduced and intensively studied by Milnor \cite{Mi}. Similarly, we define the subclass $\wcal_{\dcal}$ of $\dcal$ by
	$$
	\wcal_{\dcal} = \{A \in \dcal \ | \ A\ \text{has the $\dcal$-homotopy type of a cofibrant diffeological space} \}.
	$$
	We also define the subclass $\vcal_\dcal$ of $\dcal$ by
	$$
	\vcal_\dcal = \Set{A \in \dcal |
		\begin{array}{l}
		id : A \longrightarrow R\widetilde{A} \text{ is a weak equivalence in } \dcal
		\end{array}
		}.
	$$
	
	\begin{prop}\label{W} (1) If $A$ is in $\wcal_{\dcal}$, then $\widetilde{A}$ is in $\wcal_{\ccal^0}$.\\ \\
		(2) If $A$ is in $\wcal_\dcal$, then $A$ is in $\vcal_\dcal$.\\ \\
		(3) The following conditions are equivalent:
		\begin{itemize}
			\item[$(i)$] $A$ is in $\vcal_{\dcal}$;
			\item[$(ii)$] $\text{The inclusion } S^\dcal A \longhookrightarrow S\widetilde{A} \text{ is a weak equivalence in } \scal$;
			\item[$(iii)$] The natural homomorphism
			$$
			\pi^{\dcal}_{p}(A, a) \longrightarrow \pi_{p}(\widetilde{A}, a)
			$$
			is an isomorphism for any $a \in A$ and any $p \geq 0$.
		\end{itemize}	
	\end{prop}
	
	\begin{proof} (1) The result follows immediately from Lemma \ref{Quillenpairs} and \cite[Proposition 2.21]{origin}.\\ \\
		(2) We prove the result in two steps.\\ \\
		{\sl Step 1: The case where $A$ is cofibrant.} Since $\widetilde{\cdot}: \dcal \rightleftarrows \ccal^{0}: R$ is a pair of Quillen equivalences (Theorem \ref{Quillenequiv}) and $\ccal^{0}_{f} = \ccal^{0}$ holds, $id: A \longrightarrow R\widetilde{A}$ is a weak equivalence in $\dcal$.\\ \\
		
		\noindent {\sl Step 2. The case where $A$ is in $\wcal_{\dcal}$.} Since $A$ is in $\wcal_{\dcal}$, there are a cofibrant diffeological space $A'$ and a $\dcal$-homotopy equivalence $f: A' \longrightarrow A$. By \cite[Proposition 2.21]{origin}, $\widetilde{f}: \widetilde{A'} \longrightarrow \widetilde{A}$ is a $\ccal^{0}$-homotopy equivalence. Since $R$ preserves products and $id: I \longrightarrow R\tilde{I}$ is smooth (Lemma \ref{adjoint}(2)), $R\widetilde{f}: R\widetilde{A'} \longrightarrow R\widetilde{A}$ is a $\dcal$-homotopy equivalence. Thus, we have the commutative diagram in $\dcal$
		$$
		\begin{tikzcd}
		A' \arrow[d,swap,"id"]\arrow[r,"f"] \arrow[r,swap,"\simeq"] & A \arrow[d,"id"]\\
		R\widetilde{A}'\arrow[r,"R\widetilde{f}"]\arrow[r,swap,"\simeq"] & R\widetilde{A}.
		\end{tikzcd}
		$$
		Hence, $id : A \longrightarrow R\widetilde{A}$ is a weak equivalence by Step 1.\\ \\
		(3) $(i) \Longleftrightarrow (ii)$ From the definition of a weak equivalence in $\dcal$ (Definition \ref{WFC}) and the equality $S^\dcal R = S$ (Lemma \ref{adjoint}(2)), we see that $(i)$ is equivalent to $(ii)$.\\
		$(ii)\Longleftrightarrow(iii)$ Recall that $S^\dcal A$ and $S\widetilde{A}$ are Kan complexes (Theorem \ref{originmain} and Remark \ref{arc}). Since there are natural isomorphisms
		$$
		\pi^{\dcal}_{p}(A, a) \xrightarrow[\cong]{} \pi_{p}(S^{\dcal}A, a),
		$$
		$$
		\pi_{p}(\widetilde{A}, a) \xrightarrow[\cong]{} \pi_{p}(S\widetilde{A}, a)	
		$$
		(Theorem \ref{homotopygp}), the equivalence $(ii)\Longleftrightarrow(iii)$ is obvious (see, for example, \cite[Section 3]{K}).
	\end{proof}
	
	We give three types of diffeological spaces which does not have the $\dcal$-homotopy type of a cofibrant object.

	\begin{lem}\label{nonW}
		Let $X$ be an arc-generated space in $\wcal_{\ccal^0}$ which is compact. Then, $\pi_1 (X)$ and $H_1 (X)$ are finitely generated. Further, if $X$ is $1$-connected, then $\pi_i(X)$ and $H_i (X)$ are finitely generated for any $i > 0$.
		\begin{proof}
			Since $X \in \wcal_{\ccal^0}$, there exist a $CW$-complex $Z$ and a homotopy commutative diagram in $\ccal^0$
			\begin{center}
				\begin{tikzpicture} 
				\node [below] at (0,0.25) {\begin{tikzcd}   
					X \arrow[r, "f"]
					& Z \arrow[r, "g"]
					& X .
					\end{tikzcd}};
				\draw[->] (-1.5,-.5) -- (-1.5,-.75)--(1.4,-.75)-- (1.4,-0.5);   
				\node at (0,-1.) {$1_X$};
				\end{tikzpicture}
			\end{center}
			Since $X$ is compact, we choose a finite subcomplex $Z_0$ of $Z$ such that the diagram above restricts to the homotopy commutative diagram in $\ccal^0$
			\begin{center}
				\begin{tikzpicture}
					\node[below] at (0, 0.25) {
						\begin{tikzcd}
							X \arrow[r, "f"] &
							Z_0 \arrow[r, "g"] &
							X.
						\end{tikzcd}
					};
					\draw[->] (-1.5, -.5)--(-1.5, -.75)--(1.45, -.75)--(1.45,-.5);
					\node at (0,-1.){$1_X$};
				\end{tikzpicture}
			\end{center}
			Thus, we have the commutative diagram in the category of groups
			\begin{center}
				\begin{tikzpicture}
					\node[below] at (0, 0.25) {
						\begin{tikzcd}
						\pi_1 (X) \arrow[r, "f_\sharp"] &
						\pi_1 (Z_0) \arrow[r, "g_\sharp"] &
						\pi_1 (X)
						\end{tikzcd}
					};
					\draw[->] (-2.2, -.5)--(-2.2, -.75)--(1.9, -.75)--(1.9,-.5);
					\node at (0,-1.){$1_{\pi_1 (X)}$};
				\end{tikzpicture}
			\end{center}
			with $\pi_1 (Z_0)$ finitely generated, which shows that $\pi_1 (X)$, and hence $H_1 (X)$ is finitely generated.
			\par\indent
			If $X$ is $1$-connected, we may assume that $Z$ is a $CW$-complex whose $1$-skeleton is a singleton (\cite[pp. 74-75]{MayAT}). Since $Z_0$ is a $1$-connected finite complex, we thus see that $\pi_i (Z_0)$ and $H_i (Z_0)$ are finitely generated (\cite[Theorem 4.5.2]{MP}), and hence that $\pi_i (X)$ and $H_i (X)$ are also finitely generated.
		\end{proof}
	\end{lem}
	
	\begin{exa}\label{Hawaiian}
		The $n$-dimensional {\it Hawaiian earring} $A^n$ is defined to be the diffeological subspace of $\rbb^{n+1}$
		$$
		\overset{\infty}{\underset{m=1}{\cup}} \ \Set{ (x_0, \cdots , x_n) \in \rbb^{n+1} |
			\begin{array}{l}
				(x_0 - \frac{1}{\sqrt[n]{m}})^2 + x_1^2 + \cdots + x_n^2 = (\frac{1}{\sqrt[n]{m}})^2
			\end{array} 
		}.
		$$
		\begin{itemize}
			\item[(1)] $A^n$ does not have the $\dcal$-homotopy type of a cofibrant object.
			\item[(2)] The natural homomorphism
			$$
			\pi_n^\dcal (A^n) \longrightarrow \pi_n (\widetilde{A^n})
			$$
			is not surjective.
		\end{itemize}
		\begin{proof}
			\begin{itemize}
				\item[(1)] By the argument of \cite[Example 3.12]{CW} (see also \cite[p. 18]{KM}), we see that the underlying topology of $A^n$ is just the sub-topology of $\rbb^{n+1}$ and that $\widetilde{A^n}$ is homomorphic to the topological subspace of $\rbb^{n+1}$
				$$
				\overset{\infty}{\underset{m=1}{\cup}} \ \Set{ (x_0, \cdots , x_n) \in \rbb^{n+1} |
					\begin{array}{l}
					(x_0 - \frac{1}{m})^2 + x_1^2 + \cdots + x_n^2 = (\frac{1}{m})^2
					\end{array} 
				}.
				$$
				Thus, the result follows from \cite{EKK}, \cite{EKK2}, Lemma \ref{nonW}, and Proposition \ref{W}(1).
				\item[(2)] We can prove the result by an argument similar to that in \cite[Example 3.12]{CW}. \qed
			\end{itemize}
		\color{white} \end{proof}
	\end{exa}
	
	A diffeological space $A$ is called {\it arcwise connected} if $\pi_0^\dcal (X) = 0$; recall that $A$ is arcwise connected if and only if $\widetilde{A}$ is arcwise connected (\cite{I0}).
	
	\begin{lem}\label{homog}
		Let $G$ be an arcwise connected diffeological group, and $H$ a diffeological subgroup of $G$. If $H$ is dense in $\widetilde{G}$ and the inclusion $H \longhookrightarrow G$ is not a weak equivalence in $\dcal$, then $G / H$ does not have the $\dcal$-homotopy type of a cofibrant object.
		\begin{proof}
			By the assumption, $\widetilde{G / H}$ is an indiscrete space (\cite[Lemma 2.14]{origin}), and hence, contractible.
			\par\indent
			On the other hand, we see that $\pi_\ast^\dcal (G/H) \neq 0$ from the homotopy exact sequence of the sequence
			$$
			H \longhookrightarrow G \longrightarrow G/H
			$$
			(\cite[8.11 and 8.21]{I}). Thus, we obtain the result by Proposition \ref{W}(2).
		\end{proof}
	\end{lem}
	
	For an irrational number $\theta$, the {\it irrational torus} $T_\theta^2$ {\it of slope} $\theta$ is defined to be the quotient diffeological group $T^2 / \rbb_\theta$, where $T^2 = \rbb^2 / \zbb^2$ is the usual $2$-torus and $\rbb_\theta$ is the image of the injective homomorphism $\rbb \longrightarrow T^2$ sending $x$ to $[x, \theta x]$.
	
	\begin{exa}\label{quotient}
		The quotient diffeological groups $\rbb / \qbb$ and the irrational torus $T_\theta$ do not have the $\dcal$-homotopy type of a cofibrant object.
		\begin{proof}
			The result is immediate from Lemma \ref{homog}.
		\end{proof}
	\end{exa}
	
	\begin{rem}\label{nonisom}
		It was pointed out by Christensen-Wu that the natural homomorphism $\pi_1^\dcal (A, a) \longrightarrow \pi_1 (\widetilde{A}, a)$ may fail to be injective or surjective (\cite[Examples 3.20 and 3.12]{CW}).
	\end{rem}

	\begin{lem}\label{R}
		Let $X$ be an arc-generated space and $K$ a simplicial set. Then, any smooth map $f : RX \longrightarrow |K|_\dcal$ is locally constant.
		\begin{proof}
			We may assume that $X$ is arcwise connected. Suppose that there exists a nonconstant smooth map $f : RX \longrightarrow |K|_\dcal$. Then, we choose $x_0, x_1 \in X$ and a continuous curve $d : \rbb \longrightarrow X$ satisfying the following conditions:
			\begin{itemize}
				\item $f(x_0) \neq f(x_1)$.
				\item $d((-\infty, 0]) = x_0$, $d([1, \infty)) = x_1$.
			\end{itemize}
			By the construction, the composite
			$$
			\rbb \xrightarrow{\ \ d\ \ } RX \xrightarrow{\ \ f\ \ } |K|_\dcal
			$$
			is a nonconstant smooth map.
			\par\indent
			Let us use the presentation of the $CW$-complex $|K|_\dcal$ in the proof of Lemma \ref{realization}. Setting
			$$
			m = \text{min}\ \{ n \mid (f \circ d)(\rbb) \subset |K^n|_\dcal \},
			$$
			we obtain the nonconstant smooth map
			$$
			\rbb \xrightarrow{\ \ f \circ d\ \ } |K^m|_\dcal
			$$
			(\cite[Lemma 9.6]{origin}). Since each map $\mathring{\Delta}_\lambda^m \longrightarrow |K^m|_\dcal$ is a $\dcal$-embedding, we regard $\mathring{\Delta}_\lambda^m$ as a diffeological subspace of $|K^m|_\dcal$. Note that $(f \circ d)^{-1} (\mathring{\Delta}_\lambda^m)$ is a nonempty open set of $\rbb$ for some $\lambda \in NK_m$. Thus, we replace $d : \rbb \longrightarrow X$ so that the nonconstant smooth map
			$$
			\rbb \xrightarrow{\ \ f \circ d\ \ } |K^m|_\dcal
			$$
			correstricts to $\mathring{\Delta}_\lambda^m$. By the definitin of the functor $R$, the composite
			$$
			\rbb \xrightarrow{\ \ c\ \ } \rbb \xrightarrow{\ \ f \circ d\ \ } \mathring{\Delta}_\lambda^m
			$$
			is smooth for any continuous curve $c$, which is a contradiction; recall that $\mathring{\Delta}_\lambda^m \cong \mathring{\Delta}_{\rm sub}^m$ (Lemma \ref{delta}(3)).
		\end{proof}
	\end{lem}

	\begin{exa}\label{RM}
		Let $X$ be an arc-generated space.
		\begin{itemize}
			\item[(1)] $RX$ is in $\vcal_\dcal$.
			\item[(2)] $RX$ has the $\dcal$-homotopy type of a cofibrant object if and only if $X$ has the homotopy type of a discrete set. In particular, if $X$ is arcwise connected but not contractible, then $RX$ does not have the $\dcal$-homotopy type of a cofibrant object.
		\end{itemize}
		\begin{proof}
			\begin{itemize}
				\item[(1)] The statement is obvious since $\tilde{\cdot} \circ R = Id_{\ccal^0}$.
				\item[(2)] We have only to prove the former half of the statement.
			\end{itemize}
			\noindent($\Leftarrow$) Note that $f, g : X \longrightarrow Y$ are $\ccal^0$-homotopic, then $Rf, Rg : RX \longrightarrow RY$ are $\dcal$-homotopic (see the argument in the proof of Proposition \ref{W}(2)). Then, the implication is obvious.\\
			($\Rightarrow$) By the assumption, there exist a simplicial set $K$ and a $\dcal$-homotopy equivalence $f : RX \longrightarrow |K|_\dcal$ (Theorem \ref{Quillenequiv}). By the equality $\tilde{\cdot} \circ R = Id_{\ccal^0}$, Lemma \ref{adjoint}(2), and \cite[Proposition 2.21]{origin}, $\widetilde{f} : X \longrightarrow |K|$ is a $\ccal^0$-homotopy equivalence. Since $\widetilde{f}$ is locally constant (Lemma \ref{R}), $|K|$ is $\ccal^0$-homotopy equivalent to a discrete set.
		\end{proof}
	\end{exa}
	
	Set $\widetilde{\ }\wcal_{\ccal^0} = \{ A \in \dcal \mid \widetilde{A} \in \wcal_{\ccal^0} \}$. From Proposition \ref{W}, we have the following inclusion relations:
	$$
	\begin{tikzcd}
	& & \vcal_\dcal\arrow[dr, hook] & & \\
	\wcal_\dcal\arrow[r, hook] & \vcal_\dcal\cap\widetilde{\ }\wcal_{\ccal^0}\arrow[ur, hook]\arrow[dr, hook] & & \vcal_\dcal\cup\widetilde{\ }\wcal_{\ccal^0}\arrow[r, hook] & \dcal. \\
	& & \widetilde{\ }\wcal_{\ccal^0}\arrow[ur, hook] & &
	\end{tikzcd}
	$$
	From Examples \ref{Hawaiian}, \ref{quotient}, and \ref{RM}, we see that all the inclusions in the above diagram are genuine. 
	One of the most important problems in the homotopical algebra of diffeological spaces is the following:
	\begin{problem}\label{PW}
		Find a sufficient large class of $C^{\infty}$-manifolds (in the sense of Fr\"{o}licher-Kriegl) having the $\dcal$-homotopy type of a cofibrant object. 
	\end{problem}
	We address this problem in the forthcoming paper.\\ \\

\section{Chen spaces and diffeological spaces}
In this section, we construct a model structure on the category $\ccal h$ of Chen spaces and prove Theorem \ref{ChD}. We begin by recalling the definition of a Chen space.
\par\indent
In this section, a {\sl convex set} means a convex subset of $\rbb^{n}$ with nonempty interior for some $n$, and a {\sl parametrization} of a set $X$ means a (set-theoretic) map $p: C \longrightarrow X$, where $C$ is a convex set. See \cite[p. 5793]{BH} for the precise definition of a smooth map between convex sets.

\begin{defn}\label{Chensp}
	(1) A {\sl Chen space} is a set $X$ together with a specified set $Ch_X$ of parametrizations of $X$ satisfying the following conditions:
	\begin{itemize}
		\item[(i)](Covering)  Every constant parametrization $p:C\longrightarrow X$ is in $Ch_X$.
		\item[(ii)](Locality) Let $p :C\longrightarrow X$ be a parametrization such that there exists an open cover $\{C_i\}$ of $C$ consisting of convex sets and satisfying $p|_{C_i}\in Ch_X$. Then $p$ is in $Ch_X$
		\item[(iii)](Smooth compatibility) Let $p:C\longrightarrow X$ be in $Ch_X$. Then, for every convex set $C'$ and every smooth map $F: C' \longrightarrow C$, $p\circ F$ is in $Ch_X$.
	\end{itemize}
	Elements of $Ch_{X}$ are called {\sl plots}.\\
	\par\noindent
	(2) Let $X = (X, Ch_{X})$ and $Y = (Y, Ch_{Y})$ be Chen spaces, and let $f: X \longrightarrow Y$ be a (set-theoretic) map. We say that $f$ is {\sl smooth} if for any $p \in Ch_{X}$, $f\circ p \in Ch_{Y}$. Then, Chen spaces and smooth maps form the category $\mathcal{C}h$.
\end{defn}
For a Chen space $A = (A, Ch_{A})$, the {\sl underlying topological space} $\widetilde{A}$ is defined to be the set $A$ endowed with the final topology for $Ch_{A}$. Then, the category $\mathcal{C}h$ shares the convenient properties (1)-(3) in Proposition \ref{conven}. More precisely, we have the following result.

\begin{lem}\label{categoryCh}
	$(1)$ The category $\mathcal{C}h$ has initial and final structures with respect to the underlying set functor. In particular, $\mathcal{C}h$ is complete and cocomplete. Further, the class of $\ccal h$-embeddings (i.e., injective initial morphisms) is closed under pushouts and transfinite composites.\\
	$(2)$ The category $\mathcal{C}h$ is cartesian closed.\\
	$(3)$ The underlying set functor $\mathcal{C}h \longrightarrow Set$
	is factored as the underlying topological space functor 
	$\widetilde{\cdot}:\mathcal{C}h \longrightarrow \czero$
	followed by the underlying set functor
	$\czero \longrightarrow Set$.
	Further, the functor
	$\widetilde{\cdot}:\mathcal{C}h \longrightarrow \czero$
	has a right adjoint
	$R:\czero \longrightarrow \mathcal{C}h$.
\end{lem}
\begin{proof}
	The result is shown by arguments similar to those in the case of the category $\dcal$ (see \cite[Propositions 2.1 and 2.13]{origin}).
\end{proof}

Using Lemma \ref{categoryCh}, we can define the standard $p$-simplices $\Delta^p$ in $\mathcal{C}h$, and verify Axioms 1-4 for $\Delta^{p}$ (see Section 2) in the same manner as in the case of $\dcal$. Thus, we define the singular functor $S^{\mathcal{C}h}:\mathcal{C}h \longrightarrow \scal$ in an obvious manner, and then introduce a model structure on the category $\mathcal{C}h$ in the following theorem.
\begin{thm}\label{modelCh}
	Define a map $f :X\longrightarrow Y$ in $\mathcal{C}h$ to be
	\begin{itemize}
		\item[$(1)$]
		a weak equivalence if $S^{\mathcal{C}h} f:S^{\mathcal{C}h} X\longrightarrow S^{\mathcal{C}h} Y$ is a weak equivalence in the category of simplicial sets,
		\item[$(2)$]
		a fibration if the map $f$ has the right lifting property with respect to the inclusions $\Lambda^p_k \longhookrightarrow\Delta^p$ for all $p>0$ and $0\leq k\leq p$, and
		\item[$(3)$]
		a cofibration if the map $f$ has the left lifting property with respect to all maps that are both fibrations and weak equivalences.
	\end{itemize}
	With these choices, $\mathcal{C}h$ is a compactly generated model category whose object is always fibrant.
\end{thm}
\begin{proof}
	Since the proof of Theorem \ref{originmain} is constructed using only properties (1)-(3) of the category $\dcal$ and Axioms 1-4 for the standard simplices (see \cite{origin}), it applies to the case of the category $\mathcal{C}h$ as well.
\end{proof}

If we define the smooth homotopy groups $\pi^{\mathcal{C}h}_{p}(X, x)$ of a pointed Chen space $(X, x)$ in the same manner as in the diffeological case (\cite[Section 3.1]{CW} or \cite[Chapter 5]{I}), the result analogous to Theorem \ref{homotopygp} also holds.

By the same arguments as those in the proof of Theorem \ref{Quillenequiv}, we have the following result.
\begin{thm}\label{analogue}
	\begin{itemize}
		\item[(1)]
		$|\ |_{\ccal h}: \scal \rightleftarrows \ccal h: S^{\ccal h}$ is a pair of Quillen equivalences.
		\item[(2)]
		$\tilde{\cdot}: \ccal h \rightleftarrows \ccal^{0}: R$ is a pair of Quillen equivalences.
	\end{itemize}
\end{thm}

Let us recall the adjoint pair between the categories $\ccal h$ amd $\dcal$ introduced by Stacey \cite{St} (see also \cite{BH}).
\par\indent
To distinguish the two notions of a parametrization, we call parametrizations defined in Sections 2 and 6 open-parametrizations and convex-parametrizations respectively.
\par\indent
The functor $\mathfrak{S}_\mathfrak{o} : \ccal h \longrightarrow \dcal$ is defined to assign to a Chen space $A = (A, Ch_A)$ the set $A$ endowed with the diffeology
$$
D_{\mathfrak{S}_\mathfrak{o}A} = \text{\{Chen smooth open-parametrizations\}}.
$$
The functor $\mathfrak{Ch}^\sharp : \dcal \longrightarrow \ccal h$ is defined to assign to a diffeological space $X = (X, D_X)$ the set $X$ endowed with the Chen structure 
$$
Ch_{\mathfrak{Ch}^\sharp X} = \text{\{diffeologically smooth convex-parametrizations\}}.
$$
(Use \cite[Theorem 24.5]{KM} to verify condition (iii) in Definition \ref{Chensp}.) Observe that $\mathfrak{S}_\mathfrak{o} \circ \mathfrak{Ch}^\sharp = Id_\dcal$ and that $id : A \longrightarrow \mathfrak{Ch}^\sharp \mathfrak{S}_\mathfrak{o} A$ is Chen smooth (\cite[Theorem 24.5]{KM}). Then, it is easily seen that
$$
\mathfrak{S}_\mathfrak{o} : \ccal h \rightleftarrows \dcal : \mathfrak{Ch}^\sharp
$$
is an adjoint pair.
\par\indent
We can also construct a left adjoint to $\mathfrak{S}_\mathfrak{o} : \ccal h \longrightarrow \dcal$.
\par\indent
The functor $\mathfrak{Ch}^\flat : \dcal \longrightarrow \ccal h$ is defined to assign to a diffeological space $A$ the set $A$ endowed with the Chen structure generated by $D_A = \{ p : U \longrightarrow A \}$. Observe that $\mathfrak{S}_\mathfrak{o} \circ \mathfrak{Ch}^\flat = Id_\dcal$ and that $id : \mathfrak{Ch}^\flat \mathfrak{S}_\mathfrak{o} X \longrightarrow X$ is Chen smooth. Then, it is easily seen that
$$
\mathfrak{Ch}^\flat : \dcal \rightleftarrows \ccal h : \mathfrak{S}_\mathfrak{o}
$$
is an adjoint pair.
\par\indent
\begin{lem}\label{deltapreserve} 
	$\mathfrak{S}_\mathfrak{o}\Delta^p_{\ccal h} = \Delta^p_\dcal$ holds for $p \geq 0$, where $\Delta^p_{\ecal}$ denotes the standard $p$-simplex in $\ecal = \ccal h, \dcal$.
	\begin{proof}
		Let $\rbb_\ecal$ denote the real line $\rbb$ viewed as an object of $\ecal = \ccal h, \dcal$. $\mathfrak{S}_\mathfrak{o} \rbb_{\ccal h} = \rbb_\dcal$ holds obviously. Since we have the adjoint pairs $(\mathfrak{S}_\mathfrak{o}, \mathfrak{Ch}^\sharp)$ and $(\mathfrak{Ch}^\flat, \mathfrak{S}_\mathfrak{o})$ which are compatable with the underlying set functors, $\mathfrak{S}_\mathfrak{o}$ preserves initial and final structures, and limits and colimits (\cite[Proposition 8.7.4]{FK}). Thus, we have $\mathfrak{S}_\mathfrak{o} \Delta_{\ccal h}^p = \Delta_\dcal^p$ (see Definition \ref{simplices}).
	\end{proof}
\end{lem}

In the following, an adjoint pair $F : \acal \rightleftarrows \bcal : G$ is also denoted by $\acal \xrightharpoonup{<F,G>} \bcal$.
\begin{proof}[Proof of Theorem \ref{ChD}.] Since $\mathfrak{S}_\mathfrak{o}$ is a left adjoint, we see from Lemma \ref{deltapreserve} that $\mathfrak{S}_\mathfrak{o} |K|_{\ccal h} = |K|_\dcal$ holds for $K \in \scal$. Thus, the composite
	$$
	\scal \xrightharpoonup{<|\ |_{\ccal h}, S^{\ccal h}>} \ccal h \xrightharpoonup{<\mathfrak{S}_\mathfrak{o}, \mathfrak{Ch}^\sharp>} \dcal
	$$
	is just
	$$
	\scal \xrightharpoonup{<|\ |_\dcal, S^\dcal>} \dcal .
	$$
	By Theorems \ref{Quillenequiv} and \ref{analogue}, and Lemma \ref{Quillenpairreduction}, we have only to show that $<\mathfrak{S}_\mathfrak{o}, \mathfrak{Ch}^\sharp>$ is a Quillen pair.
	\par\indent
	For a subset $A$ of $\Delta^p$, $A_\ecal$ denotes the set $A$ endowed with the initial structure for the inclusion into $\Delta_\ecal^p \ (\ecal = \ccal h, \dcal)$. Since $\mathfrak{S}_\mathfrak{o}$ preserves initial structures, we see from Lemma \ref{deltapreserve} that $\mathfrak{S}_\mathfrak{o} A_{\ccal h} = A_\dcal$. Thus, by the adjointness of $(\mathfrak{S}_\mathfrak{o}, \mathfrak{Ch}^\sharp)$, the lifting problem in $\dcal$
	$$
	\begin{tikzcd}
	A_{\dcal} \arrow{r} & X \arrow{d}{f}\\
	\Delta^p_{\dcal} \arrow[hookleftarrow]{u} \arrow[dashed]{ur} \arrow{r}& Y
	\end{tikzcd}
	$$
	is equivalent to the lifting problem in $\ccal h$
	$$
	\begin{tikzcd}
	A_{\ccal h} \arrow{r} & \mathfrak{Ch}^{\sharp}X \arrow{d}{f}\\
	\Delta^p_{\ccal h} \arrow[hookleftarrow]{u} \arrow[dashed]{ur} \arrow{r}& \mathfrak{Ch}^{\sharp}Y.
	\end{tikzcd}
	$$
	This shows that $\mathfrak{Ch}^\sharp$ preserves fibrations and trivial fibrations (Lemma \ref{generating}), and hence, that $(\mathfrak{S}_{\mathfrak{o}}, \mathfrak{Ch}^\sharp)$ is a Quillen pair.
\end{proof}

The relevant pairs of Quillen equivalences are put together in the commutative diagram in the following proposition.\\

\begin{prop}\label{adjunction}
	We have the commutative diagram consisting of pairs of Quillen equivalences
	\begin{equation*}
		\begin{tikzcd}
			& \arrow[swap, rightharpoondown]{dl}{<|\ |_{\ccal h}, S^{\ccal h}>} \scal \arrow[harpoon]{dr}{<|\ |_\dcal, S^\dcal>} & \\
			\ccal h \arrow[harpoon]{rr}{<\mathfrak{S}_{\mathfrak{o}}, \mathfrak{Ch}^\sharp>} \arrow[swap, rightharpoondown]{dr}{<\tilde{\cdot}, R>} & & \arrow[harpoon]{dl}{<\tilde{\cdot}, R>} \dcal \\
		& \ccal^0 . &
		\end{tikzcd}
	\end{equation*}
	
	\begin{proof}
		In Theorems \ref{Quillenequiv}, \ref{analogue}, and \ref{ChD}, we have shown that the five adjoint pairs in the diagram above are pairs of Quillen equivalences. The commutativity of the upper triangle is shown in the proof of Theorem \ref{ChD}. From \cite[Lemma 2.12]{origin}, we see that $\widetilde{\mathfrak{S}_\mathfrak{o}A} = \widetilde{A}$ holds for $A \in \ccal h$, and hence that the lower triangle is commutative.
	\end{proof}
\end{prop}

\begin{rem}\label{WCh}
	Define the subclass $\wcal_{\ccal h}$ of $\ccal h$ by 
	$$
	\wcal_{\ccal h} = \{ A \in \ccal h \mid A \text{ has the } \ccal h \text{-homotopy type of a cofibrant object} \}
	$$
	and consider the pair of Quillen equivalence
	$$
	\tilde{\cdot} : \ccal h \rightleftarrows \ccal^0 : R.
	$$
	Then, the results analogous to those in Section 5 hold for $\wcal_{\ccal h}$.
	\par\indent
	Consider the pair of Quillen equivalences
	$$
	\mathfrak{S}_\mathfrak{o} : \ccal h \rightleftarrows \dcal : \mathfrak{Ch}^\sharp .
	$$
	Noticing that $\mathfrak{S}_\mathfrak{o}$ preserves products and unit intervals (Lemma \ref{deltapreserve}), we see that if $A \in \wcal_{\ccal h}$, then $\mathfrak{S}_\mathfrak{o}A$ is in $\wcal_\dcal$.
	
\end{rem}


\end{document}